\documentclass[a4paper,12pt,reqno]{amsart}

\usepackage{amsmath}
\usepackage{amssymb}
\usepackage{amsfonts}
\usepackage{graphicx}
\usepackage[colorlinks]{hyperref}
\renewcommand\eqref[1]{(\ref{#1})} 

\graphicspath{ {images/} }
\title[Weighted $L^{p}$-Hardy and $L^{p}$-Rellich inequalities]{Weighted $L^{p}$-Hardy and $L^{p}$-Rellich inequalities with boundary terms
	on stratified Lie groups}
\author{Michael Ruzhansky}
\address{Imperial College London, London, UK}
\email{m.ruzhansky@imperial.ac.uk}

\author{Bolys Sabitbek}
\address{Institute of Mathematics and Mathematical Modeling \\
	and Al-Farabi Kazakh National University, Almaty, Kazakhstan}

\email{sabitbek@math.kz}

\author{Durvudkhan Suragan}
\address{Institute of Mathematics and Mathematical Modeling, Almaty, Kazakhstan}
\email{suragan@math.kz}


\newtheoremstyle{theorem}
{10pt}          
{10pt}  
{\sl}  
{\parindent}     
{\bf}  
{. }    
{ }    
{}     
\theoremstyle{theorem}

\numberwithin{equation}{section}
\theoremstyle{plain}
\newtheorem{thm}{Theorem}[section]
\newtheorem{prop}[thm]{Proposition}
\newtheorem{cor}[thm]{Corollary}
\newtheorem{lem}[thm]{Lemma}

\theoremstyle{definition}

\newtheorem{rem}[thm]{Remark}

\newcommand{\G}{\mathbb G}
\newtheoremstyle{defi}
{10pt}          
{10pt}  
{\rm}  
{\parindent}     
{\bf}  
{. }    
{ }    
{}     
\theoremstyle{defi}

\begin{document}

\thanks{The first author was supported by the EPSRC Grant 
EP/R003025/1 and by 
 the Leverhulme Research Grant RPG-2017-151. No new data was collected or generated during the course of this research.}
\date{\today}

	\begin{abstract}
		In this paper, generalised weighted $L^p$-Hardy, $L^p$-Caffarelli-Kohn-Ni\-ren\-berg, and $L^p$-Rellich inequalities with boundary terms  are obtained on stratified Lie groups. As consequences, most of the Hardy type inequalities and Heisenberg-Pauli-Weyl type uncertainty principles on stratified groups are recovered.  Moreover, a weighted $L^2$-Rellich type inequality with the boundary term is obtained.
	\end{abstract}
	
	\maketitle
	\section{Introduction}
	
	Let $\mathbb{G}$ be a stratified Lie group (or a homogeneous Carnot group), with dilation structure $\delta_{\lambda}$ and Jacobian generators $X_{1},\ldots,X_{N}$, so that $N$ is the dimension of the first stratum of $\mathbb{G}$.  We refer to \cite{Folland}, or to the recent books  \cite{BLU} or \cite{FR} for extensive discussions of stratified Lie groups and their properties. 
Let $Q$ be the homogeneous dimension of $\G$.	
	The sub-Laplacian on $\mathbb{G}$ is given by
	\begin{equation}\label{sublap}
	\mathcal{L}=\sum_{k=1}^{N}X_{k}^{2}.
	\end{equation}
It was shown by Folland \cite{Folland} that the sub-Laplacian has a unique fundamental solution $\varepsilon$,
	\begin{equation*}
		\mathcal{L} \varepsilon = \delta,
	\end{equation*}
	where $\delta$ denotes the Dirac distribution with singularity at the neutral element $0$ of $\G$. The fundamental solution $\varepsilon(x,y)=\varepsilon(y^{-1}x)$ is homogeneous of degree $-Q+2$ and can be written in the form
	\begin{equation}\label{EQ:Fuso}
		\varepsilon(x,y) = [d(y^{-1}x)]^{2-Q},
	\end{equation}
for some homogeneous $d$ which is called the $\mathcal{L}$-gauge.
Thus, the $\mathcal{L}$-gauge is a symmetric homogeneous (quasi-) norm on the stratified group $\G = (\mathbb{R}^n,\circ,\delta_{\lambda})$, that is,
	\begin{itemize}
		\item $d(x) >0$ if and only if $x \neq 0$,
		\item $d(\delta_{\lambda}(x))=\lambda d(x)$ for all $\lambda>0$ and $x \in \G$,
		\item $d(x^{-1})=d(x)$ for all $x \in \G$.
	\end{itemize}
	We also recall that the standard Lebesque measure $dx$ on $\mathbb R^{n}$ is the Haar measure for $\mathbb{G}$ (see, e.g. \cite[Proposition 1.6.6]{FR}).
	The left invariant vector field $X_{j}$ has an explicit form and satisfies the divergence theorem,
	see e.g. \cite{FR} for the derivation of the exact formula: more precisely, we can write
	\begin{equation}\label{Xk0}
	X_{k}=\frac{\partial}{\partial x'_{k}}+
	\sum_{l=2}^{r}\sum_{m=1}^{N_{l}}a_{k,m}^{(l)}(x',...,x^{(l-1)})
	\frac{\partial}{\partial x_{m}^{(l)}},
	\end{equation}
	with $x=(x',x^{(2)},\ldots,x^{(r)})$, where $r$ is the step of $\G$ and 
	$x^{(l)}=(x^{(l)}_1,\ldots,x^{(l)}_{N_l})$ are the variables in the $l^{th}$ stratum,
	see also \cite[Section 3.1.5]{FR} for a general presentation.
	The horizontal gradient is given by
	$$\nabla_{\G}:=(X_{1},\ldots, X_{N}),$$
	and the horizontal divergence is defined by
	$${\rm div}_{\G} v:=\nabla_{\G}\cdot v.$$
	The horizontal $p$-sub-Laplacian is defined by 
	\begin{equation}\label{pLap}
	\mathcal{L}_{p}f:={\rm div}_{\G}(|\nabla_{\G}f|^{p-2}\nabla_{\G}f),\quad 1<p<\infty,
	\end{equation}
	and we will write
	$$|x'|=\sqrt{x'^{2}_{1}+\ldots+x'^{2}_{N}}$$ for the Euclidean norm on $\mathbb{R}^{N}.$
	
	Throughout this paper $\Omega\subset\mathbb{G}$ will be an admissible domain, that is, an open set $\Omega\subset\mathbb{G}$ is called an {\em admissible domain} if it is bounded and if its boundary $\partial\Omega$ is piecewise smooth and simple i.e., it has no self-intersections.
	The condition for the boundary to be simple amounts to $\partial\Omega$ being orientable.
	
	We now recall the divergence formula in the form of \cite[Proposition 3.1]{RS17b}. 
	Let $f_{k}\in C^{1}(\Omega)\bigcap C(\overline{\Omega}),\,k=1,\ldots,N$.
	Then for each $k=1,\ldots,N,$ we have 
	\begin{equation}\label{EQ:S1}
	\int_{\Omega}X_{k}f_{k}dz=
	\int_{\partial\Omega}f_{k} \langle X_{k},dz\rangle.
	\end{equation}
	Consequently, we also have 
	\begin{equation}\label{EQ:S2}
	\int_{\Omega}\sum_{k=1}^{N}X_{k}f_{k}dz=
	\int_{\partial\Omega}
	\sum_{k=1}^{N} f_{k}\langle X_{k},dz\rangle.
	\end{equation}
	Using the divergence formula analogues of Green's formulae were obtained in \cite{RS17b} for general Carnot groups and in \cite{Ruzhansky-Suragan:squares} for more abstract settings (without the group structure), for another formulation see also \cite{CGH08}.
	
	The analogue of Green's first formula for the sub-Laplacian was given in \cite{RS17b} in the following form: if $ v \in C^1(\Omega)\cap C(\overline{\Omega})$ and $u \in  C^2(\Omega)\cap C^1(\overline{\Omega})$, then
	\begin{equation} \label{g1}
	\int_{\Omega}\left((\mathcal{\widetilde{\nabla}}v) u +v\mathcal{L}u\right) dz=\int_{\partial\Omega}v\langle \mathcal{\widetilde{\nabla} }u,dz\rangle,
	\end{equation}
	where 
	\begin{equation*}
		\widetilde{\nabla} u = \sum_{k=1}^{N} (X_k u)X_k,
	\end{equation*}
	and 
	\begin{equation*} 
		\int_{\partial \Omega} \sum_{k=1}^{N} \langle v X_k u X_k,dz \rangle= \int_{\partial \Omega} v \langle \widetilde{\nabla} u, dz \rangle.
	\end{equation*}
	Rewriting \eqref{g1} we have
	$$\int_{\Omega}\left((\mathcal{\widetilde{\nabla} }u) v+u\mathcal{L}v\right)dz=\int_{\partial\Omega}u\langle \mathcal{\widetilde{\nabla} }v,dz\rangle,$$
	$$\int_{\Omega}\left((\mathcal{\widetilde{\nabla} }v) u+v\mathcal{L}u\right)dz=\int_{\partial\Omega}v\langle \mathcal{\widetilde{\nabla} }u,dz\rangle.$$
	By using $(\mathcal{\widetilde{\nabla} }u) v=(\mathcal{\widetilde{\nabla} }v) u$ and subtracting one identity for the other we get Green's second formula for the sub-Laplacian:
	\begin{equation}\label{g2}
	\int_{\Omega}(u\mathcal{L}v-v\mathcal{L}u)dz
	=\int_{\partial\Omega}(u\langle\widetilde{\nabla}  v,dz\rangle-v\langle \widetilde{\nabla}  u,dz\rangle).
	\end{equation}
	
	It is important to note that the above Green's formulae also hold for the fundamental solution of the sub-Laplacian as in the case of the fundamental solution of the (Euclidean) Laplacian since both have the same behaviour near the singularity $z=0$ (see \cite[Proposition 4.3]{ARS}).
	
	Weighted Hardy and Rellich inequalities in different related contexts have been recently considered in \cite{KY} and \cite{GKY}. For the general importance of such inequalities we can refer to \cite{Balinsky}. Some boundary terms have appeared in \cite{WZ}.
	
	The main aim of this paper is to give the generalised weighted $L^p$-Hardy and $L^p$-Rellich type inequalities on stratified groups. 
	In Section \ref{SEC:Lpw}, we present a weighted $L^p$-Caffarelli-Kohn-Nirenberg type inequality with boundary term on stratified group $\G$, which implies, in particular, the weighted $L^p$-Hardy type inequality. As consequences of those inequalities, we recover most of the known Hardy type inequalities and Heisenberg-Pauli-Weyl type uncertainty principles on stratified group $\G$ (see \cite{Ruzhansky-Suragan:ProcA} for discussions in this direction). 
	In Section \ref{SEC:BSS3}, a weighted $L^p$-Rellich type inequality is investigated. Moreover, a weighted $L^2$-Rellich type inequality with the boundary term is obtained together with its  consequences.  
	
	Usually, unless we state explicitly otherwise, the functions $u$ entering all the inequalities are complex-valued.

	\section{Weighted $L^p$-Hardy type inequalities with boundary terms and their consequences}
	\label{SEC:Lpw}
	
In this section we derive several versions of the $L^p$ weighted Hardy inequalities. 	
	
	\subsection{Weighted $L^p$-Cafferelli-Kohn-Nirenberg type inequalities with boundary terms}
	We first present the following weighted $L^p$-Cafferelli-Kohn-Nirenberg type inequalities with boundary terms on the stratified Lie group $\G$ and then discuss their consequences. The proof of Theorem \ref{CKN_thm} is analogous to the proof of Davies and Hinz \cite{DavHinz}, but is now carried out  in the case of the stratified Lie group $\G$. The boundary terms also give new addition to the Euclidean results in \cite{DavHinz}. The classical Caffarelli-Kohn-Nirenberg inequalities in the Euclidean setting were obtained in \cite{CKN_bib}.
	
	Let $\G$ be a stratified group with $N$ being the dimension of the first stratum, and let $V$ be a real-valued function in $L_{loc}^1(\Omega)$ with partial derivatives of order up to 2 in $L_{loc}^1(\Omega)$, and such that $\mathcal{L} V$ is of one sign. Then we have:

	\begin{thm}\label{CKN_thm}
		Let $\Omega$ be an admissible domain in the  stratified group $\mathbb{G}$, and let $V$ be a real-valued function such that $\mathcal{L} V<0$ holds a.e. in $\Omega$. Then for any complex-valued  $u \in C^2(\Omega)\cap C^1(\overline{\Omega})$, and all $1<p<\infty$, we have the inequality 
		
		\begin{equation}\label{CKN_term}
		\left\| |\mathcal{L} V|^{\frac{1}{p}} u\right\|_{L^p(\Omega)}^p  \leq p  \left\| \frac{|\nabla_{\G} V|}{|\mathcal{L} V|^{\frac{p-1}{p}}}|\nabla_{\G} u| \right\|_{L^p(\Omega)} \left\||\mathcal{L}V|^{\frac{1}{p}} u  \right\|_{L^p(\Omega)}^{p-1} -\int_{\partial \Omega} |u|^p\langle\widetilde{\nabla} V,dx\rangle.
		\end{equation}
	\end{thm}
	Note that if $u$ vanishes on the boundary $\partial \Omega$, then \eqref{CKN_term} extends the Davies and Hinz result \cite{DavHinz} to the weighted $L^p$-Hardy type inequality on stratified groups:
	\begin{equation}\label{Hardy}
	\left\| |\mathcal{L} V|^{\frac{1}{p}} u\right\|_{L^p(\Omega)}  \leq p  \left\| \frac{|\nabla_{\G} V|}{|\mathcal{L} V|^{\frac{p-1}{p}}}|\nabla_{\G} u| \right\|_{L^p(\Omega)}, \quad 1<p<\infty.
	\end{equation}
	
	\begin{proof}[Proof of Theorem \ref{CKN_thm}]
		Let  $\upsilon_{\epsilon}:=(|u|^2 + \epsilon^2)^{\frac{1}{2}}-\epsilon$. Then $\upsilon_{\epsilon}^p \in C^2({\Omega}) \cap C^1({\overline{\Omega}})$ and using Green's first formula \eqref{g1} and the fact that $\mathcal{L} V < 0$ we get 
		\begin{align*}
			\int_{\Omega}|\mathcal{L} V|\upsilon_{\epsilon}^p dx &= - \int_{\Omega} \mathcal{L} V \upsilon_{\epsilon}^p dx \\
			&= \int_{\Omega} (\widetilde{\nabla}V) \upsilon_{\epsilon}^p dx - \int_{\partial \Omega}\upsilon_{\epsilon}^p \langle\widetilde{\nabla} V,dx\rangle  \\
			& = \int_{\Omega} \nabla_{\G} V \cdot \nabla_{\G} \upsilon_{\epsilon}^p dx - \int_{\partial \Omega}\upsilon_{\epsilon}^p \langle\widetilde{\nabla} V,dx\rangle \\
			& \leq \int_{\Omega} |\nabla_{\G} V| |\nabla_{\G} \upsilon_{\epsilon}^p| dx - \int_{\partial \Omega}\upsilon_{\epsilon}^p \langle\widetilde{\nabla} V,dx\rangle \\
			& = p \int_{\Omega} \left(\frac{|\nabla_{\G} V|}{|\mathcal{L} V|^{\frac{p-1}{p}}}\right) |\mathcal{L} V|^{\frac{p-1}{p}} \upsilon_{\epsilon}^{p-1}|\nabla_{\G} \upsilon_{\epsilon}|dx- \int_{\partial \Omega}\upsilon_{\epsilon}^p \langle\widetilde{\nabla} V,dx\rangle,
		\end{align*}
		where $(\widetilde{\nabla}u)v = \nabla_{\G} u\cdot \nabla_{\G} v$. We have
		\begin{equation*}
			\nabla_{\G} \upsilon_{\epsilon} = (|u|^2 +\epsilon^2)^{-\frac{1}{2}} |u|\nabla_{\G} |u|,
		\end{equation*}
		since $0 \leq \upsilon_{\epsilon} \leq |u|$. Thus,
		\begin{equation*}
			\upsilon_{\epsilon}^{p-1} |\nabla_{\G} \upsilon_{\epsilon}| \leq |u|^{p-1}|\nabla_{\G}|u||.
		\end{equation*}
		On the other hand, 	let $u(x)=R(x)+iI(x)$, where $R(x)$ and $I(x)$ denote the real and imaginary parts of $u$. We can restrict to the set where $u\neq 0$. Then  we have
		\begin{equation}\label{eq1}
		(\nabla_{\G} |u|)(x) = 
		\frac{1}{|u|} (R(x)\nabla_{\G} R(x) + I(x)\nabla_{\G} I(x)) \quad \text{if} \quad u \neq 0.
		\end{equation}
		Since 
		\begin{equation}\label{eq3}
		\left| \frac{1}{|u|} (R\nabla_{\G} R+I \nabla_{\G} I)\right|^2 \leq |\nabla_{\G} R|^2 + |\nabla_{\G} I|^2,
		\end{equation}
		we get that  $|\nabla_{\G} |u||\leq |\nabla_{\G} u|$ a.e. in $\Omega$.
		Therefore,
		\begin{multline*}
			\int_{\Omega}|\mathcal{L} V|\upsilon_{\epsilon}^p dx \leq p \int_{\Omega} \left(\frac{|\nabla_{\G} V|}{|\mathcal{L} V|^{\frac{p-1}{p}}} |\nabla_{\G} u|\right)|\mathcal{L} V|^{\frac{p-1}{p}}|u|^{p-1} dx- \int_{\partial \Omega}\upsilon_{\epsilon}^p \langle\widetilde{\nabla} V,dx\rangle \\
			 \leq p \left( \int_{\Omega} \left(\frac{|\nabla_{\G} V|^p}{|\mathcal{L} V|^{(p-1)}} |\nabla_{\G} u|^p\right)dx \right)^{\frac{1}{p}} \left(\int_{\Omega}|\mathcal{L} V||u|^p dx\right )^{\frac{p-1}{p}}- \int_{\partial \Omega}\upsilon_{\epsilon}^p \langle\widetilde{\nabla} V,dx\rangle,
		\end{multline*}
		where we have used H\"older's inequality in the last line. Thus, when  $\epsilon \rightarrow 0$, we obtain \eqref{CKN_term}.
	\end{proof}

	\subsection{Consequences of Theorem \ref{CKN_thm}}
	
	As consequences of Theorem \ref{CKN_thm}, we can derive the horizontal $L^p$-Caffarelli-Kohn-Nirenberg type inequality with the boundary term on the stratified group $\G$ which also gives another proof of $L^p$-Hardy type inequality, and also yet another proof of the Badiale-Tarantello conjecture \cite{BT} (for another proof see e.g. \cite{RS17a} and references therein).
	
	\subsubsection{Horizontal $L^p$-Caffarelli--Kohn--Nirenberg inequalities with the boundary term}
	\begin{cor}\label{cor1}
		Let $\Omega$ be an admissible domain in a stratified group $\G$ with $N\geq 3$ being dimension of the first stratum, and let $\alpha, \beta \in \mathbb{R}$. Then for all $u \in C^2(\Omega\backslash \{x'=0\})\cap C^1(\overline{\Omega} \backslash \{x'=0\})$, and any $1<p<\infty$, we have
		\begin{equation}\label{ex1}
			\frac{|N-\gamma |}{p} \left\| \frac{u}{|x'|^{\frac{\gamma}{p}}}\right\|^p_{L^p(\Omega)} \leq \left\| \frac{\nabla_{\G} u}{|x'|^{\alpha}} \right\|_{L^p(\Omega)} \left\|\frac{u}{|x'|^{\frac{\beta}{p-1}}} \right\|^{p-1}_{L^p(\Omega)} 
			 -\frac{1}{p} \int_{\partial \Omega} |u|^p\langle\widetilde{\nabla} |x'|^{2-\gamma},dx\rangle,
		\end{equation}
		for $2<\gamma<N$ with $\gamma = \alpha +\beta+1, $ and where $|\cdot|$ is the Euclidean norm on $\mathbb{R}^{N}$.
		In particular, if $u$ vanishes on the boundary $\partial \Omega$, we have
		\begin{equation}\label{123}
		\frac{|N-\gamma |}{p} \left\| \frac{u}{|x'|^{\frac{\gamma}{p}}}\right\|^p_{L^p(\Omega)} \leq \left\| \frac{\nabla_{\G} u}{|x'|^{\alpha}} \right\|_{L^p(\Omega)} \left\|\frac{u}{|x'|^{\frac{\beta}{p-1}}} \right\|^{p-1}_{L^p(\Omega)}.
		\end{equation}
	\end{cor}
	\begin{proof}[Proof of Corollary \ref{cor1}]
		To obtain \eqref{ex1} from \eqref{CKN_term} , we take $V= |x'|^{2-\gamma}$. Then
		$$
			|\nabla_{\G} V| = |2-\gamma||x'|^{1-\gamma},\qquad
			|\mathcal{L} V| =|(2-\gamma)(N-\gamma)| |x'|^{-\gamma},
		$$
		and observe that 
			$\mathcal{L} V = (2-\gamma)(N-\gamma) |x'|^{-\gamma}<0.$
		To use \eqref{CKN_term} we calculate
		\begin{equation*}
			\left\| |\mathcal{L} V|^{\frac{1}{p}} u\right\|_{L^p(\Omega)}^p
			= |(2-\gamma)(N-\gamma)|\left\| \frac{u}{|x'|^{\frac{\gamma}{p}}} \right\|_{L^p(\Omega)}^p,
		\end{equation*}
		\begin{equation*}
			\left\| \frac{|\nabla_{\G} V|}{|\mathcal{L} V|^{\frac{p-1}{p}}} \nabla_{\G} u \right\|_{L^p(\Omega)}  = \frac{|2-\gamma|}{|(2-\gamma)(N-\gamma)|^{\frac{p-1}{p}}} \left\| \frac{|\nabla_{\G} u|}{|x'|^{\frac{\gamma-p}{p}}} \right\|_{L^p(\Omega)},
		\end{equation*}
		\begin{equation*}
			\left\||\mathcal{L}V|^{\frac{1}{p}} u  \right\|_{L^p(\Omega)}^{p-1} = |(2-\gamma)(N-\gamma)|^{\frac{p-1}{p}}\left\| \frac{u}{|x'|^{\frac{\gamma}{p}}}\right\|_{L^p(\Omega)}^{p-1}.
		\end{equation*}
		Thus, \eqref{CKN_term} implies
		$$
			\frac{|N-\gamma|}{p}\left\| \frac{u}{|x'|^{\frac{\gamma}{p}}} \right\|_{L^p(\Omega)}^p \leq \left\| \frac{\nabla_{\G} u}{|x'|^{\frac{\gamma-p}{p}}} \right\|_{L^p(\Omega)}\left\| \frac{u}{|x'|^{\frac{\gamma}{p}}}\right\|_{L^p(\Omega)}^{p-1} 
			- \frac{1}{p}\int_{\partial \Omega} |u|^p\langle\widetilde{\nabla} |x'|^{2-\gamma},dx\rangle.
		$$
		If we denote $\alpha = \frac{\gamma-p}{p}$  and $\frac{\beta}{p-1}=\frac{\gamma}{p}$, we get \eqref{ex1}.
	\end{proof}
	
	\subsubsection{Badiale-Tarantello conjecture} Theorem \ref{CKN_thm} also gives a new proof of the generalised Badiale-Tarantello conjecture \cite{BT} (see, also \cite{RS17a}) on the optimal constant in Hardy inequalities in $\mathbb{R}^n$ with weights taken with respect to a subspace.
	
	\begin{prop}\label{B-T}
		Let $x=(x',x'') \in \mathbb{R}^N \times \mathbb{R}^{n-N}$, $1\leq N \leq n$, $2<\gamma<N$ and $\alpha, \beta \in \mathbb{R}$. Then for any $u \in C_0^{\infty}(\mathbb{R}^n \backslash \{x'=0\} )$ and all $1<p<\infty$, we have
		\begin{equation}\label{BT_eq}
		\frac{|N-\gamma |}{p} \left\| \frac{u}{|x'|^{\frac{\gamma}{p}}}\right\|^p_{L^p(\mathbb{R}^n)} \leq \left\| \frac{\nabla u}{|x'|^{\alpha}} \right\|_{L^p(\mathbb{R}^n)} \left\|\frac{u}{|x'|^{\frac{\beta}{p-1}}} \right\|^{p-1}_{L^p(\mathbb{R}^n)},
		\end{equation}
		where $\gamma = \alpha+\beta+1$ and $|x'|$ is the Euclidean norm $\mathbb{R}^N$. If $\gamma \neq N$ then the constant $	\frac{|N-\gamma |}{p}$ is sharp.
	\end{prop}
	The proof of Proposition \ref{B-T} is similar to Corollary \ref{cor1}, so we sketch it only very briefly.
	\begin{proof}[Proof of Proposition \ref{B-T}]
		Let us take $V= |x'|^{2-\gamma}$. We observe that
		$
			\Delta V = (2-\gamma)(N-\gamma) |x'|^{-\gamma}<0,
		$
		as well as 
		$|\nabla V| = |2-\gamma||x'|^{(1-\gamma)}$ and 
$|\Delta V| = |(2-\gamma)(N-\gamma)| |x'|^{-\gamma}$.
Then \eqref{CKN_term} with 
		\begin{equation*}
			\left\| |\Delta V|^{\frac{1}{p}} u\right\|_{L^p(\mathbb{R}^n)}^p
			= |(2-\gamma)(N-\gamma)|\left\| \frac{u}{|x'|^{\frac{\gamma}{p}}} \right\|_{L^p(\mathbb{R}^n)}^p,
		\end{equation*}
		\begin{equation*}
			\left\| \frac{|\nabla V|}{|\Delta V|^{\frac{p-1}{p}}} \nabla u \right\|_{L^p(\mathbb{R}^n)}  = \frac{|2-\gamma|}{|(2-\gamma)(N-\gamma)|^{\frac{p-1}{p}}} \left\| \frac{\nabla u}{|x'|^{\frac{\gamma-p}{p}}} \right\|_{L^p(\mathbb{R}^n)},
		\end{equation*}
		\begin{equation*}
			\left\||\Delta V|^{\frac{1}{p}} u  \right\|_{L^p(\mathbb{R}^n)}^{p-1} = |(2-\gamma)(N-\gamma)|^{\frac{p-1}{p}}\left\| \frac{u}{|x'|^{\frac{\gamma}{p}}}\right\|_{L^p(\mathbb{R}^n)}^{p-1},
		\end{equation*}
		and denoting $\alpha = \frac{\gamma-p}{p}$  and $\frac{\beta}{p-1}=\frac{\gamma}{p}$, implies \eqref{BT_eq}.
	\end{proof}
	In particular, if we take $\beta =(\alpha+1)(p-1)$ and $\gamma=p(\alpha+1)$, then \eqref{BT_eq} implies 
	\begin{equation}\label{3.16}
	\frac{|N-p(\alpha+1) |}{p} \left\| \frac{u}{|x'|^{\alpha+1}}\right\|_{L^p(\mathbb{R}^n)} \leq \left\| \frac{\nabla u}{|x'|^{\alpha}}  \right\|_{L^p(\mathbb{R}^n)},
	\end{equation}
	where $1<p< \infty$, for all $u \in C_0^{\infty}(\mathbb{R}^n \backslash \{x'=0\})$, $\alpha \in \mathbb{R}$, with sharp constant. When $\alpha =0$, $1<p<N$ and $2\leq N\leq n$, the inequality \eqref{3.16} implies that
	\begin{equation}
	\left\| \frac{u}{|x'|}  \right\|_{L^p(\mathbb{R}^n)} \leq \frac{p}{N-p} \left\| \nabla u \right\|_{L^p(\mathbb{R}^n)},
	\end{equation}
which given another proof of the Badiale-Tarantello conjecture from \cite[Remark 2.3]{BT}.
	
	
	\subsubsection{The local Hardy type inequality on $\G$.}  As another consequence of Theorem \ref{CKN_thm} we obtain the local Hardy type inequality with the boundary term, with $d$ being the $\mathcal{L}$-gauge as in \eqref{EQ:Fuso}. 
	\begin{cor}\label{fund_ineq}
		Let $\Omega \subset \G$ with $0 \notin \partial \Omega$ be an admissible domain in a stratified group $\G$ of homogeneous dimension $Q\geq 3.$ Let $0>\alpha > 2-Q$. Let $u \in C^{1}(\Omega \backslash \{0\})\cap C(\overline{\Omega}\backslash \{0\})$. Then we have
		\begin{multline}\label{COR123}
			\frac{|Q+\alpha-2|}{p} \left\|d^{\frac{\alpha-2}{p}} |\nabla_{\G} d|^{\frac{2}{p}} u \right\|_{L^p(\Omega)} \leq \left\| d^{\frac{p+\alpha-2}{p}} |\nabla_{\G}d|^{\frac{2-p}{p}} |\nabla_{\G}u| \right\|_{L^p(\Omega)} \\
			- \frac{1}{p}\left\| d^{\frac{\alpha-2}{p}} |\nabla_{\G} d|^{\frac{2}{p}} u\right\|^{1-p}_{L^p(\Omega)} \int_{\partial \Omega} d^{\alpha-1} |u|^p \langle\widetilde{\nabla} d,dx\rangle. 
		\end{multline}
	\end{cor}
	This extends the local Hardy type inequality that was obtained in \cite{RS17b}  for $p=2$:
	\begin{multline}
		\frac{|Q+\alpha-2|}{2} \left\|d^{\frac{\alpha-2}{2}} |\nabla_{\G} d| u \right\|_{L^2(\Omega)} \leq \left\| d^{\frac{\alpha}{2}}  |\nabla_{\G}u| \right\|_{L^2(\Omega)} \\
		-\frac{1}{2} \left\| d^{\frac{\alpha-2}{2}} |\nabla_{\G} d| u\right\|_{L^2(\Omega)}^{-1} \int_{\partial \Omega} d^{\alpha-1} |u|^2 \langle\widetilde{\nabla} d,dx\rangle .
	\end{multline}
	
	\begin{proof}[Proof of Corollary \ref{fund_ineq}]
		First, we can multiply both sides of the inequality \eqref{CKN_term} by $\left\| |\mathcal{L}V|^{\frac{1}{p}} u \right\|_{L^p(\Omega)}^{1-p}$, so that we have
		\begin{multline}\label{Hardy_term}
			\left\| |\mathcal{L}V|^{\frac{1}{p}}u \right\|_{L^p(\Omega)} \leq p \left\| \frac{|\nabla_{\G}V|}{|\mathcal{L}V|^{\frac{p-1}{p}}}|\nabla_{\G} u|\right\|_{L^p(\Omega)}
			- \left\||\mathcal{L} V|^{\frac{1}{p}} u\right\|_{L^p(\Omega)}^{1-p} \int_{\partial \Omega}|u|^p \langle\widetilde{\nabla}V,dx\rangle.
		\end{multline}
		Now, let us take $V=d^{\alpha}$. We have
		\begin{multline*}
			\mathcal{L} d^{\alpha} = \nabla_{\G}(\nabla_{\G} \varepsilon^{\frac{\alpha}{2-Q}}) = \nabla_{\G}\left(\frac{\alpha}{2-Q}\varepsilon^{\frac{\alpha+Q-2}{2-Q}}\nabla_{\G}\varepsilon \right) \\
			 = \frac{\alpha(\alpha+Q-2)}{(2-Q)^2} \varepsilon^{\frac{\alpha-4+2Q}{2-Q}}|\nabla_{\G} \varepsilon|^2 + \frac{\alpha}{2-Q}\varepsilon^{\frac{\alpha+Q-2}{2-Q}} \mathcal{L} \varepsilon.
		\end{multline*}
		Since $\varepsilon$ is the fundamental solution of $\mathcal{L}$, we have
		\begin{equation*}
			\mathcal{L} d^{\alpha} = \frac{\alpha(\alpha+Q-2)}{(2-Q)^2} \varepsilon^{\frac{\alpha-4+2Q}{2-Q}}|\nabla_{\G} \varepsilon|^2= \alpha(\alpha+Q-2)d^{\alpha-2}|\nabla_{\G}d|^2.
		\end{equation*}
		We can observe that $\mathcal{L}d^{\alpha} <0$, and also the identities 		\begin{equation*}
			\left\||\mathcal{L}d^{\alpha}|^{\frac{1}{p}} u \right\|_{L^p(\Omega)} = \alpha^{\frac{1}{p}} |Q+\alpha-2|^{\frac{1}{p}} \left\| d^{\frac{\alpha-2}{p}} |\nabla_{\G}d|^{\frac{2}{p}} u \right\|_{L^p(\Omega)},
		\end{equation*}
		\begin{equation*}
			\left\| \frac{|\nabla_{\G}d^{\alpha}|}{|\mathcal{L}d^{\alpha}|^{\frac{p-1}{p}}} |\nabla_{\G} u| \right\|_{L^p(\Omega)} = \alpha^{\frac{1}{p}}|Q+\alpha-2|^{\frac{1-p}{p}} \left\| d^{\frac{\alpha-2+p}{p}} |\nabla_{\G}d|^{\frac{2-p}{p}} |\nabla_{\G}u| \right\|_{L^p(\Omega)},
		\end{equation*}
		\begin{align*}
			\left\| |\mathcal{L}d^{\alpha}|^{\frac{1}{p}} u \right\|^{1-p}_{L^p(\Omega)} \int_{\partial \Omega} |u|^p\langle\widetilde{\nabla} d^{\alpha},dx\rangle =  \alpha^{\frac{1}{p}} |Q+\alpha-2|^{\frac{1-p}{p}} \left\| d^{\frac{\alpha-2}{p}} |\nabla_{\G}d|^{\frac{2}{p}} u \right\|^{1-p}_{L^p(\Omega)}\\
			\int_{\partial \Omega} d^{\alpha-1}|u|^p\langle\widetilde{\nabla} d,dx\rangle.
		\end{align*}
		Using \eqref{Hardy_term} we arrive at
		\begin{align*}
			\frac{|Q+\alpha-2|}{p} \left\|d^{\frac{\alpha-2}{p}} |\nabla_{\G} d|^{\frac{2}{p}} u \right\|_{L^p(\Omega)} \leq \left\| d^{\frac{p+\alpha-2}{p}} |\nabla_{\G}d|^{\frac{2-p}{p}} |\nabla_{\G}u| \right\|_{L^p{\Omega}} \\
			-\frac{1}{p} \left\| d^{\frac{\alpha-2}{p}} |\nabla_{\G} d|^{\frac{2}{p}} u\right\|^{1-p}_{L^p(\Omega)} \int_{\partial \Omega} d^{\alpha-1} |u|^p \langle\widetilde{\nabla} d,dx\rangle ,\nonumber
		\end{align*}
	which implies \eqref{COR123}.
	\end{proof}
	
	
	\subsection{Uncertainty type principles}
	The inequality \eqref{Hardy_term} implies the following Heisen\-berg-Pauli-Weyl type uncertainty principle on stratified groups.
	\begin{cor}\label{UP}
		Let $\Omega \subset \G$ be admissible domain in a stratified group $\G$ and let $V \in C^2(\Omega)$ be real-valued. Then for any complex-valued function $u \in C^2(\Omega)\cap C^1(\overline{\Omega})$ we have
		\begin{multline}\label{uncert_term}
			\left\| |\mathcal{L}V|^{-\frac{1}{p}} u \right\|_{L^p(\Omega)}\left\| \frac{|\nabla_{\G}V|}{|\mathcal{L}V|^{\frac{p-1}{p}}} |\nabla_{\G} u| \right\|_{L^p(\Omega)} \\
			\geq \frac{1}{p} \left\| u \right\|^2_{L^p(\Omega)} 
			+ \frac{1}{p}\left\| |\mathcal{L}V|^{-\frac{1}{p}} u \right\|_{L^p(\Omega)}\left\| |\mathcal{L}V|^{\frac{1}{p}} u \right\|^{1-p}_{L^p(\Omega)}\int_{\partial \Omega} |u|^p \langle\widetilde{\nabla} V,dx\rangle.
		\end{multline}
		In particular, if $u$ vanishes on the boundary $\partial \Omega$, then we have
		\begin{equation}\label{uncert}
		\left\| |\mathcal{L}V|^{-\frac{1}{p}} u \right\|_{L^p(\Omega)}\left\| \frac{|\nabla_{\G}V|}{|\mathcal{L}V|^{\frac{p-1}{p}}} |\nabla_{\G} u| \right\|_{L^p(\Omega)} \geq \frac{1}{p} \left\| u \right\|^2_{L^p(\Omega)}.
		\end{equation}
	\end{cor}
	\begin{proof}[Proof of Corollary \ref{UP}]
	By using the extended H\"older inequality and \eqref{Hardy_term} we have
		\begin{align*}
			&\left\| |\mathcal{L}V|^{-\frac{1}{p}} u \right\|_{L^p(\Omega)}\left\| \frac{|\nabla_{\G}V|}{|\mathcal{L}V|^{\frac{p-1}{p}}} |\nabla_{\G} u| \right\|_{L^p(\Omega)} \\
			&\geq \frac{1}{p} \left\| |\mathcal{L}V|^{-\frac{1}{p}} u \right\|_{L^p(\Omega)} \left\| |\mathcal{L}V|^{\frac{1}{p}} u \right\|_{L^p(\Omega)} 
			+ \frac{1}{p} \left\| |\mathcal{L}V|^{-\frac{1}{p}} u \right\|_{L^p(\Omega)}\left\| |\mathcal{L}V|^{\frac{1}{p}} u \right\|^{1-p}_{L^p(\Omega)}\int_{\partial \Omega} |u|^p \langle\widetilde{\nabla} V,dx\rangle,\nonumber\\
			&\geq \frac{1}{p} \left\|  |u|^2 \right\|_{L^{\frac{p}{2}}(\Omega)} \nonumber  
			+ \frac{1}{p}\left\| |\mathcal{L}V|^{-\frac{1}{p}} u \right\|_{L^p(\Omega)}\left\| |\mathcal{L}V|^{\frac{1}{p}} u \right\|^{1-p}_{L^p(\Omega)}\int_{\partial \Omega} |u|^p \langle\widetilde{\nabla} V,dx\rangle.\nonumber\\
			&= \frac{1}{p} \left\|  u \right\|^2_{L^p(\Omega)} \nonumber  
			+ \frac{1}{p}\left\| |\mathcal{L}V|^{-\frac{1}{p}} u \right\|_{L^p(\Omega)}\left\| |\mathcal{L}V|^{\frac{1}{p}} u \right\|^{1-p}_{L^p(\Omega)}\int_{\partial \Omega} |u|^p \langle\widetilde{\nabla} V,dx\rangle,\nonumber
		\end{align*} 
	proving \eqref{uncert_term}.
	\end{proof}
	\medskip
	By setting $V =|x'|^{\alpha}$ in the inequality \eqref{uncert},
	we recover the Heisenberg-Pauli-Weyl type uncertainty principle on stratified groups as in  \cite{Ozawa} and \cite{Ruzhansky-Suragan:squares}:
	
	\begin{align*}\label{RS_uncert}
		\left( \int_{\Omega} |x'|^{2-\alpha} |u|^p dx \right) \left( \int_{\Omega} |x'|^{\alpha+p-2} |\nabla_{\G} u|^p dx \right) \geq \left(\frac{N+\alpha-2}{p} \right)^p \left( \int_{\Omega} |u|^p dx \right)^2.
	\end{align*}
	
	In the abelian case $\G=(\mathbb{R}^n,+)$, taking $N=n\geq 3$, for $\alpha=0$ and $p=2$ this implies the classical Heisenberg-Pauli-Weyl uncertainty principle for all $u \in C^{\infty}_0(\mathbb{R}^n \backslash \{0\})$:
	\begin{equation*}
		 \left(\int_{\mathbb{R}^n} |x|^2 |u(x)|^2 dx\right)
		 \left( \int_{\mathbb{R}^n} |\nabla u(x)|^2 dx \right) 
		 \geq \left(\frac{n-2}{2} \right)^2 \left(\int_{\mathbb{R}^n} |u(x)|^2 dx \right)^2.
	\end{equation*}
	
	By setting $V =d^{\alpha}$ in the inequality \eqref{uncert}, we obtain another uncertainty type principle:
	\begin{align*}
		\left(\int_{\Omega} \frac{|u|^p}{d^{\alpha-2}|\nabla_{\G}d|^2} dx \right) \left(\int_{\Omega} d^{\alpha+p-2}|\nabla_{\G}d|^{2-p}|\nabla_{\G}u|^p dx\right) \geq \left(\frac{Q+\alpha-2}{p} \right)^p \left( \int_{\Omega} |u|^p dx\right)^2;
	\end{align*}
	taking $p=2$ and $\alpha=0$ this yields
	\begin{equation*}
		\left(\int_{\Omega} \frac{d^{2}}{|\nabla_{\G}d|^2}|u|^2 dx \right) \left(\int_{\Omega} |\nabla_{\G}u|^2 dx\right) \geq \left(\frac{Q-2}{2} \right)^2 \left( \int_{\Omega} |u|^2 dx\right)^2.
	\end{equation*}  
	
	
	\section{Weighted $L^p$-Rellich type inequalities}
	\label{SEC:BSS3}
	
	In this section we establish weighted Rellich inequalities with boundary terms. We consider first the $L^2$ and then the $L^p$ cases. 	The analogous $L^2$-Rellich inequality on $\mathbb{R}^n$ was proved by Schmincke \cite{Shim} (and generalised by Bennett \cite{Bennett}).

	
	\begin{thm}\label{rellich}
		Let $\Omega$ be an admissible domain in a stratified group $\G$ with $N \geq 2$ being the dimension of the first stratum. If a real-valued function $V\in C^2(\Omega)$ satisfies $\mathcal{L} V(x) < 0$ for all $x \in \Omega$, then for  every $\epsilon >0$ we have
		\begin{align}\label{Rel_eq}
			\left\| \frac{|V|}{|\mathcal{L} V|^{\frac{1}{2}}} \mathcal{L} u \right\|^2_{L^2(\Omega)} &\geq 2\epsilon \left\| V^{\frac{1}{2}} |\nabla_{\G} u| \right\|^2_{L^2(\Omega)} + \epsilon(1-\epsilon)
			\left\| |\mathcal{L}V|^{\frac{1}{2}} u\right\|^2_{L^2(\Omega)}\nonumber\\
			&- \epsilon \int_{\partial \Omega} (|u|^2\langle \widetilde{\nabla}V,dx\rangle - V\langle\widetilde{\nabla} |u|^2,dx\rangle), 	
		\end{align}
		for all complex-valued functions $u \in C^2(\Omega)\cap C^1(\overline{\Omega})$. In particular, if $u$ vanishes on the boundary $\partial \Omega$,  we have
		\begin{equation}
		\left\| \frac{|V|}{|\mathcal{L} V|^{\frac{1}{2}}} \mathcal{L} u \right\|^2_{L^2(\Omega)} \geq 2\epsilon \left\| V^{\frac{1}{2}} |\nabla_{\G} u| \right\|^2_{L^2(\Omega)} + \epsilon(1-\epsilon)
		\left\| |\mathcal{L}V|^{\frac{1}{2}} u\right\|^2_{L^2(\Omega)}\nonumber.
		\end{equation}
		
	\end{thm}
	\begin{proof}[Proof of Theorem \ref{rellich}]
		Using Green's second identity \eqref{g2} and that $\mathcal{L} V(x) <0$ in $\Omega$, we obtain
		\begin{align*}
			\int_{\Omega} |\mathcal{L} V||u|^2dx &= -\int_{\Omega}V\mathcal{L}|u|^2 dx -  \int_{\partial \Omega} (|u|^2\langle \widetilde{\nabla}V,dx\rangle - V\langle\widetilde{\nabla} |u|^2,dx\rangle) \\ 
			&=-2\int_{\Omega}V\left({\rm Re}(\overline{u}\mathcal{L} u)+|\nabla_{\G} u|^2\right)dx - \int_{\partial \Omega} (|u|^2\langle \widetilde{\nabla}V,dx\rangle - V\langle\widetilde{\nabla} |u|^2,dx\rangle).
		\end{align*}
		Using the Cauchy-Schwartz inequality we get
		\begin{align*}
			\int_{\Omega} |\mathcal{L} V||u|^2 dx &\leq 2\left(\frac{1}{\epsilon} \int_{\Omega} \frac{|V|^2}{|\mathcal{L} V|}|\mathcal{L} u|^2 dx\right)^{\frac{1}{2}} \left(\epsilon \int_{\Omega}|\mathcal{L} V||u|^2 dx\right)^{\frac{1}{2}}\\
			&-2 \int_{\Omega}V|\nabla_{\G} u|^2 dx - \int_{\partial \Omega} (|u|^2\langle \widetilde{\nabla}V,dx\rangle - V\langle\widetilde{\nabla} |u|^2,dx\rangle) \\
			&\leq \frac{1}{\epsilon} \int_{\Omega} \frac{|V|^2}{|\mathcal{L} V|} |\mathcal{L} u|^2 dx + \epsilon \int_{\Omega}|\mathcal{L} V||u|^2dx \\
			& - 2\int_{\Omega} V |\nabla_{\G} u|^2 dx- \int_{\partial \Omega} (|u|^2\langle \widetilde{\nabla}V,dx\rangle - V\langle\widetilde{\nabla} |u|^2,dx\rangle),
		\end{align*}
yielding \eqref{Rel_eq}.
	\end{proof}
	
	\begin{cor}\label{cor3.2}
		Let $ \G$ be a  stratified group with $N$ being the dimension of the first stratum. If $\alpha>-2$ and $N>\alpha+4$ then for all $u \in C_0^{\infty}(\G \backslash \{x'=0\})$ we have
		\begin{equation}\label{6.2.3}
		\int_{\G \backslash \{x'=0\}} \frac{|\mathcal{L} u|^2}{|x'|^{\alpha}} dx \geq \frac{(N+\alpha)^2(N-\alpha-4)^2}{16}	\int_{\G \backslash \{x'=0\}} \frac{|u|^2}{|x'|^{\alpha+4}}dx.
		\end{equation}
			\end{cor}
	\begin{proof}[Proof of Corollary \ref{cor3.2}]
		Let us take $V(x)=|x'|^{-(\alpha+2)}$ in Theorem \ref{rellich}, which can be applied since $x'=0$ is not in the support of $u$. Then we have
		\begin{equation*}
			\nabla_{\G} V = -(\alpha+2)|x'|^{-\alpha-4}x',\qquad 
			\mathcal{L} V = - (\alpha+2)(N-\alpha -4)|x'|^{-(\alpha+4)}.
		\end{equation*}
		Let us set $C_{N,\alpha}:=(\alpha+2)(N-\alpha-4)$. Observing that
		\begin{equation*}
			\mathcal{L}V = -C_{N,\alpha}|x'|^{-(\alpha+4)}<0 ,
		\end{equation*}
		for $|x'| \neq 0$, it follows from \eqref{rellich} that
		\begin{align}\label{6.2.4}
			\int_{\G \backslash \{x'=0\}} \frac{|\mathcal{L} u|^2}{|x'|^{\alpha}} dx &\geq 2 C_{N,\alpha} \epsilon \int_{\G \backslash \{x'=0\}} \frac{|\nabla_{\G} u|^2}{|x'|^{\alpha+2}}dx \nonumber \\
			&+C_{N,\alpha}^2 \epsilon(1-\epsilon) \int_{\G \backslash \{x'=0\}} \frac{|u|^2}{|x'|^{\alpha+4}}dx.
		\end{align}
		To obtain \eqref{6.2.3}, let us apply the $L^p$-Hardy type inequality \eqref{Hardy} by taking $V(x)=|x'|^{\alpha+2}$ for $\alpha \in (-2,N-4)$, so that
		\begin{equation*}
			\int_{\G \backslash \{x'=0\}} \frac{|\nabla_{\G} u|^2}{|x'|^{\alpha+2}} dx \geq \frac{(N-\alpha-4)^2}{4} \int_{\G \backslash \{x'=0\}} \frac{|u|^2}{|x'|^{\alpha+4}}dx,
		\end{equation*}
		and then choosing $\epsilon =(N+\alpha)/4(\alpha+2)$ for \eqref{6.2.4}, which is the choice of $\epsilon$ that gives the maximum right-hand side.
	\end{proof}

We can now formulate the $L^p$-version of weighted $L^p$-Rellich type inequalities.
	
	\begin{thm}\label{Rellic}
		Let $\Omega$ be an admissible domain in a stratified group $\G$. If $0<V \in C(\Omega )$, $\mathcal{L} V <0$, and $\mathcal{L}(V^{\sigma})\leq 0$ on $\Omega$ for some $\sigma >1$, then for all $u \in C_0^{\infty}(\Omega)$ we have
		\begin{equation}\label{rel_eq}
		\left\| |\mathcal{L} V|^{\frac{1}{p}}u \right\|_{L^p(\Omega)} \leq \frac{p^2}{(p-1)\sigma +1} \left\| \frac{V}{|\mathcal{L}V|^{\frac{p-1}{p}}}\mathcal{L}u \right\|_{L^p(\Omega)}, \quad 1\leq p < \infty.
		\end{equation}
			\end{thm}
	
	Theorem \ref{Rellic} will follow by Lemma \ref{lem2}, by putting $C =\frac{(p-1)(\sigma-1)}{p}$ in Lemma \ref{lem1}.  
	
	\begin{lem}\label{lem1}
		Let $\Omega$ an admissible domain in a   stratified group $\G$. If $V \geq 0$, $\mathcal{L} V <0$, and there exists a constant $C\geq 0$ such that 
		\begin{equation}\label{subs}
		C \left\| |\mathcal{L} V|^{\frac{1}{p}} u \right\|_{L^p(\Omega)}^p \leq p(p-1) \left\|V^{\frac{1}{p}} |u|^{\frac{p-2}{p}}|\nabla_{\G} u|^{\frac{2}{p}}\right \|_{L^p(\Omega)}^p, \quad 1<p<\infty,
		\end{equation}
		for all $u \in C_0^{\infty}(\Omega)$, then we have
		\begin{equation}\label{subs2}
			(1+C)\left\| |\mathcal{L} V|^{\frac{1}{p}} u \right \|_{L^p(\Omega)} \leq p \left\| \frac{V}{|\mathcal{L}V|^{\frac{p-1}{p}}} \mathcal{L} u \right \|_{L^p(\Omega)},
		\end{equation}
		for all $u \in C_0^{\infty}(\Omega)$. If $p=1$ then the statement holds for $C=0$.
	\end{lem}
	
	\begin{proof}[Proof of Lemma \ref{lem1}] We can assume that $u$ is real-valued by using the following identity (see \cite[p. 176]{Davies}):
		\begin{equation*}
			\forall z \in \mathbb{C}: |z|^p= \left( \int_{-\pi}^{\pi} |\cos{\vartheta}|^p d\vartheta \right)^{-1} \int_{\pi}^{-\pi} |{\rm Re}(z)\cos{\vartheta}+{\rm Im}(z)\sin{\vartheta}|^{p} d \vartheta ,
		\end{equation*}
	which can be  proved by writing $z=r(\cos{\phi}+i\sin{\phi})$ and simplifying. 
		
		Let  $\epsilon >0$ and set $u_{\epsilon}:=(|u|^2+\epsilon^2)^{p/2}-\epsilon^p$. Then $0\leq u_{\epsilon} \in C_0^{\infty}$ and
		\begin{align*}
			\int_{\Omega}|\mathcal{L}V|u_{\epsilon} dx = - \int_{\Omega} (\mathcal{L}V)u_{\epsilon} dx = - \int_{\Omega}V \mathcal{L}u_{\epsilon} dx,		
		\end{align*}
		where
		\begin{align*}
			\mathcal{L}u_{\epsilon} =& \mathcal{L} \left( (|u|^2+\epsilon^2)^{\frac{p}{2}}-\epsilon^p \right)= \nabla_{\G} \cdot (\nabla_{\G}((|u|^2+\epsilon^2)^{\frac{p}{2}}-\epsilon^p))\\
			=& \nabla_{\G}(p(|u|^2+\epsilon^2)^{\frac{p-2}{2}} u \nabla_{\G}u)\\
			=& p(p-2)(|u|^2+\epsilon^2)^{\frac{p-4}{2}}u^2|\nabla_{\G}u|^2 +p(|u|^2+\epsilon^2)^{\frac{p-2}{2}}|\nabla_{\G}u|^2 
			+ p(|u|^2+\epsilon^2)^{\frac{p-2}{2}}u\mathcal{L}u.
		\end{align*}
		Then
		\begin{align*}
			\int_{\Omega}|\mathcal{L}V|u_{\epsilon} dx &= -\int_{\Omega}\left(p(p-2)u^2(u^2+\epsilon^2)^{\frac{p-4}{2}}+p(u^2+\epsilon^2)^{\frac{p-2}{2}}\right)V|\nabla_{\G}u|^2dx \\
			& \qquad - p\int_{\Omega} V u(u^2+\epsilon^2)^{\frac{p-2}{2}}\mathcal{L}udx.
		\end{align*}
		Hence
		\begin{align*}
			\int_{\Omega}  |\mathcal{L}V|u_{\epsilon} &+ \left( p(p-2)u^2(u^2+\epsilon^2)^{\frac{p-4}{2}}+p(u^2+\epsilon^2)^{\frac{p-2}{2}}\right)V|\nabla_{\G}u|^2 dx \\
			& \leq p\int_{\Omega} V|u|(u^2+\epsilon^2)^{\frac{p-2}{2}}|\mathcal{L}u|dx.
		\end{align*}
		When $\epsilon \rightarrow 0$, the integrand on the right is bounded by
		$V(\max|u|^2+1)^{(p-1)/2}\max|\mathcal{L}u|$ and it is integrable because $u \in C_0^{\infty}(\Omega)$, and so the integral tends to
		$\int_{\Omega}V|u|^{p-1}|\mathcal{L}u|dx$ by the dominated convergence theorem. The integrand on the left is non-negative and tends to
		$ |\mathcal{L} V||u|^p + p(p-1)V|u|^{p-2}|\nabla_{\G}u|^2 $
		pointwise, only for $u \neq 0$ when $p<2$, otherwise for any $x$. It then follows by Fatou's lemma that
		\begin{align*}
			\left\| |\mathcal{L}V|^{\frac{1}{p}} u \right\|_{L^p(\Omega)}^p + p(p-1)\left\| V^{\frac{1}{p}}|u|^{\frac{p-2}{p}}|\nabla_{\G}u|^{\frac{2}{p}} \right\|_{L^p(\Omega)}^p 
			\leq	p \left\| V^{\frac{1}{p}} |u|^{\frac{p-1}{p}} |\mathcal{L}u|^{\frac{1}{p}}\right\|_{L^p(\Omega)}^p.
		\end{align*}
		By using \eqref{subs}, followed by the H\"older inequality, we obtain
		\begin{align*}
			(1+C)\left\| |\mathcal{L}V|^{\frac{1}{p}} u\right\|_{L^p(\Omega)}^p & \leq p \left\| |\mathcal{L}V|^{(p-1)}V^{\frac{1}{p}}|u|^{\frac{p-1}{p}}|\mathcal{L}V|^{-(p-1)} |\mathcal{L} u|^{\frac{1}{p}} \right\|^p \\
			&\leq p \left\| |\mathcal{L}V|^{\frac{1}{p}} u \right\|_{L^p(\Omega)}^{p-1} \left\| \frac{|V|}{|\mathcal{L}V|^{\frac{p-1}{p}}} \mathcal{L} u \right\|_{L^p(\Omega)}.
		\end{align*}
		This implies \eqref{subs2}.
	\end{proof}
	
	\begin{lem}\label{lem2}
		Let $\Omega$ be an admissible domain in a   stratified group $\G$. If $0<V\in C(\Omega)$, $\mathcal{L} V <0$, and $\mathcal{L}V^{\sigma} \leq 0$ on $\Omega$ for some $\sigma >1$, then we have
		\begin{equation}
		(\sigma -1) \int_{\Omega} |\mathcal{L}V||u|^p dx \leq p^2 \int_{\{x\in\Omega, u(x) \neq 0\}} V|u|^{p-2}|\nabla_{\G} u|^2 dx < \infty, \quad 1<p<\infty,
		\end{equation}
		for all $u \in C_0^{\infty}(\Omega)$.
	\end{lem}
	\begin{proof}[Proof of Lemma \ref{lem2}]
		We shall use that
		\begin{equation}\label{6.3.5}
		0 \geq \mathcal{L}(V^{\sigma}) = \sigma V^{\sigma-2}\left( (\sigma-1)|\nabla_{\G} V|^2+V\mathcal{L}V\right),
		\end{equation}
		and hence
		\begin{equation*}
			(\sigma -1)|\nabla_{\G}V|^2 \leq V |\mathcal{L}V|.
		\end{equation*}
		Now we use the inequality \eqref{Hardy} for $p=2$ to get
		\begin{align}\label{6.3.6}
			(\sigma -1)\int_{\Omega} |\mathcal{L} V||u|^2 dx &\leq 4(\sigma -1)\int_{\Omega} \frac{|\nabla_{\G} V|^2}{|\mathcal{L} V|} |\nabla_{\G}u|^2 dx \nonumber \\
			& \leq 4 \int_{\Omega} V |\nabla_{\G}u|^2 dx = 4 \int_{\{x\in \Omega; u(x)\neq 0, |\nabla_{\G}u|\neq 0 \}} V |\nabla_{\G} u|^2 dx,
		\end{align}
		the last equality valid since $|\{x \in \Omega; u(x)=0, |\nabla_{\G} u| \neq 0 \}|=0$. This proves Lemma \ref{lem2} for $p=2$.
		
		For $p \neq 2$, put $v_{\epsilon}=(u^2 +\epsilon^2)^{p/4}-\epsilon^{p/2}$, and let $\epsilon \rightarrow 0$. Since $0\leq v_{\epsilon}\leq |u|^{\frac{p}{2}}$, the left-hand side of \eqref{6.3.6}, with $u$ replaced by $v_{\epsilon}$, tends to $(\sigma -1)\int_{\Omega}|\mathcal{L}V||u|^pdx$ by the dominated convergence theorem. If $u \neq 0$, then
		\begin{align*}
			|\nabla_{\G}v_{\epsilon}|^2 V =\left| \frac{p}{2}u(u^2+\epsilon^2)^{\frac{p-4}{4}}\nabla_{\G}u\right|^2 V.
		\end{align*}
		For $\epsilon \rightarrow 0$ we obtain
		\begin{equation*}
			|\nabla_{\G}u|^p V =	 \frac{p^2}{4} |u|^{p-2}|\nabla_{\G}u|^2 V.		
		\end{equation*}
		It follows as in the proof of Lemma \ref{lem1}, by using Fatou's lemma, that the right-hand side of \eqref{6.3.6} tends to
		\begin{equation*}
			p^2\int_{\{x\in\Omega; u(x) \neq 0,|\nabla_{\G}u|\neq 0\}} V |u|^{p-2}|\nabla_{\G}u|^2dx,
		\end{equation*}
		and this completes the proof.
	\end{proof}
	
	\begin{cor}\label{4.6}
		Let $\G$ be a  stratified group with $N$ being the dimension of the first stratum. Then for any  $2<\alpha <N$  and all $u \in C_0^{\infty}(\G \backslash\{x'=0\})$ we have  the inequality
		\begin{equation}\label{123}
		\int_{\G} \frac{|u|^p}{|x'|^{\alpha}} dx \leq C^p_{(N,p,\alpha)} \int_{\G} \frac{|\mathcal{L}u|^p}{|x'|^{\alpha -2p}}dx,
		\end{equation}
		where
		\begin{equation}\label{3.10}
		C_{(N,p,\alpha)} = \frac{p^2}{(N-\alpha)\left((p-1)N+\alpha -2p\right)}.
		\end{equation}
	\end{cor}
	\begin{proof}[Proof of Corollary \ref{4.6}]
		Let us choose $V=|x'|^{-(\alpha -2)}$  in Theorem \ref{Rellic}, so that
		\begin{equation*}
			\mathcal{L} V  = -(\alpha-2)(N-\alpha)|x'|^{-\alpha}, 
		\end{equation*}
		and we note that when $2<\alpha<N$, we have 
		$
			\mathcal{L} V < 0
		$
		for $|x'|\neq 0$. Now it follows from \eqref{rel_eq} that 
		\begin{align}
			(\alpha-2)^p(N-\alpha)^p \int_{\G} \frac{|u|^p}{|x'|^{\alpha}} dx \leq \frac{p^{2p}}{[(p-1)\sigma+1]^p} \int_{\G} \frac{|\mathcal{L}u|^p}{|x'|^{\alpha-2p}}dx. 
		\end{align}
		By taking $\sigma=(N-2)/(\alpha -2)$, we arrive at 
		\begin{equation*}
			\int_{\G} \frac{|u|^p}{|x'|^{\alpha}} dx \leq  \frac{p^{2p}}{(N-\alpha)^p\left((p-1)N+\alpha -2p\right)^p} \int_{\G} \frac{|\mathcal{L}u|^p}{|x'|^{\alpha -2p}}dx,
		\end{equation*}
		which proves \eqref{123}--\eqref{3.10}.
	\end{proof}
	\begin{cor}\label{Kombe}
		Let $\G$ be a stratified Lie group and let $d=\varepsilon^{\frac{1}{2-Q}}$, where $\varepsilon$ is the fundamental solution of the sub-Laplacian $\mathcal{L}$. Assume that $Q\geq 3$, $\alpha <2$, and $Q+\alpha-4>0$. Then for all $u \in C_0^{\infty}(\G \backslash \{0\})$ we have		\begin{align}\label{Rel_Kombe}
			\frac{(Q+\alpha-4)^2(Q-\alpha)^2}{16} \int_{\G} d^{\alpha-4} |\nabla_{\G} d|^2 |u|^2 dx \leq \int_{\G} \frac{d^{\alpha}}{|\nabla_{\G}d|^2} |\mathcal{L} u|^2 dx.
		\end{align}
	\end{cor}
	The inequality \eqref{Rel_Kombe} was obtained by Kombe \cite{Kombe:Rellich-Carnot-2010}, but now we get it as an immediate consequence of Theorem \ref{Rellic}. 
	\begin{proof}[Proof of Corollary \ref{Kombe}]
		Let us choose $V=d^{\alpha -2}$  in Theorem \ref{Rellic}. Then  
		\begin{equation*}
			\mathcal{L} V = (\alpha-2)(Q+\alpha-4)d^{\alpha-4}|\nabla_{\G}d|^2.
		\end{equation*}
		Note that for $Q+\alpha-4>0$ and  $\alpha <2$, we have
		$
			\mathcal{L}V<0$ for all $x\neq 0.
		$
		If $p=2$ then from \eqref{rel_eq} it follows that 
		\begin{equation*}
			(\alpha-2)^2(Q+\alpha-4)^2\int_{\G} d^{\alpha-4}|\nabla_{\G} d|^2 |u|^2 dx \leq \frac{16}{(\sigma+1)^2}\int_{\G} \frac{d^{\alpha}}{|\nabla_{\G}d|^2} |\mathcal{L}u|^2  dx.   
		\end{equation*}
		By taking  $\sigma=(Q-2\alpha+2)/(\alpha -2)$ we get 
		\begin{align*}
			\frac{(Q+\alpha-4)^2(Q-\alpha)^2}{16} \int_{\G} d^{\alpha-4} |\nabla_{\G} d|^2 |u|^2 dx \leq \int_{\G} \frac{d^{\alpha}}{|\nabla_{\G}d|^2} |\mathcal{L} u|^2 dx,
		\end{align*}
		proving inequality \eqref{Rel_Kombe}.
	\end{proof}
	\begin{rem}\label{classic Rellich}
		In the abelian case, when $\G \equiv (\mathbb{R}^n,+)$ with $d =|x|$ being the Euclidean norm, and $\alpha=0$ in inequality \eqref{Rel_Kombe}, we recover the classical Rellich inequality \cite{Rellich}.
	\end{rem}

\end{document}